\documentclass[twoside,a4paper,12pt,centertags]{amsart}
\usepackage{amsmath,amssymb,verbatim,vmargin}
\usepackage[colorlinks,citecolor=red,pagebackref,hypertexnames=false]{hyperref}

\theoremstyle{plain}
\newtheorem{thm}{Theorem}[section]
\newtheorem{lem}[thm]{Lemma}
\newtheorem{cor}[thm]{Corollary}
\newtheorem{prop}[thm]{Proposition}

\theoremstyle{definition}

\newtheorem{rem}[thm]{Remark}

\mathchardef\semic="303B
\newcommand{\R}{{\mathbb R}}
\newcommand{\N}{{\mathbb N}}
\newcommand{\C}{{\mathbb C}}
\newcommand{\Z}{{\mathbb Z}}
\newcommand{\mH}{{\mathcal H}}

\newcommand{\mX}{{\mathcal X}}

\newcommand{\mL}{{\mathcal L}}

\newcommand{\mA}{{\mathcal A}}

\newcommand{\mR}{{\mathcal R}}

\newcommand{\mS}{{\mathcal S}}

\newcommand{\mI}{{\mathcal I}}
\newcommand{\mB}{{\mathcal B}}

\DeclareMathOperator{\re}{Re}

\newcommand{\nul}{\textsf{N}}
\newcommand{\ran}{\textsf{R}}
\newcommand{\dom}{\textsf{D}}

\newcommand{\clos}[1]{\overline{#1}}
\newcommand{\closran}[1]{\overline{\ran(#1)}}

\newcommand{\barint}{\mbox{$ave \int$}}
\newcommand{\divv}{{\text{{\rm div}}}}

\newcommand{\abs}[1]{|#1|}

\newcommand{\Norm}[2]{\|#1\|_{#2}}

\def\barint_#1{\mathchoice
            {\mathop{\vrule width 6pt
height 3 pt depth -2.5pt
                    \kern -8.8pt
\intop \kern -4pt}\nolimits_{#1}}%
            {\mathop{\vrule width 5pt height
3 pt depth -2.6pt
                    \kern -6.5pt
\intop \kern -4pt}\nolimits_{#1}}%
            {\mathop{\vrule width 5pt height
3 pt depth -2.6pt
                    \kern -6pt
\intop \kern -4pt}\nolimits_{#1}}%
            {\mathop{\vrule width 5pt height
3 pt depth -2.6pt
          \kern -6pt \intop \kern -4pt}\nolimits_{#1}}}
          
          \def\bariint_#1{\mathchoice
            {\mathop{\vrule width 10pt
height 3 pt depth -2.5pt
                    \kern -12.8pt
\intop \kern -10pt\intop \kern -4pt}\nolimits_{#1}}%
            {\mathop{\vrule width 9pt height
3 pt depth -2.6pt
                    \kern -10.5pt
\intop \kern -10pt\intop \kern -4pt}\nolimits_{#1}}%
            {\mathop{\vrule width 9pt height
3 pt depth -2.6pt
                    \kern -10pt
\intop \kern -10pt\intop \kern -4pt}\nolimits_{#1}}%
            {\mathop{\vrule width 9pt height
3 pt depth -2.6pt
          \kern -10pt \intop \kern -10pt\intop \kern -4pt}
      \nolimits_{  #1}}}

\usepackage{color}

\definecolor{gr}{rgb}   {0.,   0.8,   0. } 
\definecolor{bl}{rgb}   {0.,   0.5,   1. } 
\definecolor{mg}{rgb}   {0.7,  0.,    0.7}

\makeatletter
\@namedef{subjclassname@2010}{
  \textup{2010} Mathematics Subject Classification}
\makeatother

\title{Remarks on functional calculus for perturbed first order Dirac operators}
\author{Pascal Auscher}
\address{Univ. Paris-Sud, laboratoire de Math\'ematiques, UMR 8628 du CNRS, F-91405 {\sc Orsay}} 
\email{pascal.auscher@math.u-psud.fr}

\author{Sebastian Stahlhut}
\address{Univ. Paris-Sud, laboratoire de Math\'ematiques, UMR 8628 du CNRS, F-91405 {\sc Orsay}} 
\email{sebastian.stahlhut@math.u-psud.fr}

\date{\today}

\begin{document}

\maketitle

\begin{abstract}
We make some remarks on earlier works on  $R-$bisectoriality in $L^p$ of perturbed first order differential operators by Hyt\"onen, McIntosh and Portal.  They have shown  that 
this is equivalent to bounded holomorphic functional calculus in $L^p$ for $p$ in any open interval when suitable hypotheses are made. Hyt\"onen and McIntosh then showed that $R$-bisectoriality  in $L^p$ at one value of $p$ can be extrapolated in a neighborhood of $p$. We give a different proof of this extrapolation and observe that the first proof has impact on  the splitting of the space by the kernel and range. 
\end{abstract}

\section{Introduction}

Recall that an unbounded operator $\mathcal{A}$ on a Banach space $\mX$ is called \emph{bisectorial} of angle $\omega\in[0,\pi/2)$ if  it is closed, its spectrum is contained in the closure of $S_{\omega}:=\{z\in\C;\abs{\arg(\pm z)}<\omega\},$ and one has the resolvent estimate
\begin{equation*}
  \Norm{(I+\lambda\mathcal{A})^{-1}}{\mL(\mX)}\leq C_{\omega'}\qquad\forall\ \lambda\notin S_{\omega'},\quad\forall\ \omega'>\omega.
\end{equation*}
Assuming reflexivity of $\mX$,  this implies that the domain is dense and also the  fact that the null space and the closure of the range split. More precisely, we say that the operator $\mA$  \emph{kernel/range decomposes} if $\mX=\nul(\mA) \oplus \closran \mA$ ($\oplus$ means that the sum is topological).  Here $\nul(\mA)$ denotes the kernel or null space and $\ran(\mA)$ its range, while the domain is denoted by $\dom(\mA)$. Bisectoriality in a reflexive space is stable under taking adjoints. 

For any bisectorial operator, one can define a calculus of bounded operators by  the Cauchy integral formula,
\begin{equation*}
\begin{split}
  \psi(\mathcal{A}) &:=\frac{1}{2\pi i}\int_{\partial S_{\omega'}}\psi(\lambda)(I-\frac{1}{\lambda}\mathcal{A})^{-1}\frac{d\lambda}{\lambda}, \\
  \psi &\in \Psi(S_{ \omega''}):=\{\phi\in H^{\infty}(S_{\omega''}):\phi\in O\big(\inf({\abs{z},\abs{z^{-1}}})^{\alpha}\big),\alpha>0\},
\end{split}
\end{equation*}
with $\omega''>\omega'>\omega$. If this calculus may be boundedly extended to all $\psi\in H^{\infty}(S_{\omega''})$, the space of bounded holomorphic functions in  $S_{\omega''}$  for all $\omega''>\omega$, then $\mathcal{A}$ is said to have an \emph{$H^{\infty}$-calculus} of angle $\omega$.

Assume $\mX=L^q$ of some $\sigma$-finite measure space. A closed operator $\mathcal{A}$ is called \emph{$R$-bisectorial} of angle $\omega$ if its spectrum is contained in $S_{w}$ and for all $\omega'>\omega$, there exists a constant $C>0$ such that 
$$
\left\|\left(\sum_{j=1}^k|(I+\lambda_{j} \mA)^{-1}u_{j}|^2\right)^{1/2}\right\|_{q} \le C
\left\|\left(\sum_{j=1}^k|u_{j}|^2\right)^{1/2}\right\|_{q}
$$
for all $k\in \N$, $\lambda_{1}, \ldots, \lambda_{k}\notin S_{\omega'}$ and $u_{1}, \ldots, u_{k}\in L^q$.  This is the so called $R$-boundedness criterion applied to the resolvent family. Note that the definition implies that $\mA$ is bisectorial.  This notion can be defined on any Banach space but we do need this here. 

In \cite{HMP1} and \cite{HMP2}, the equivalence between bounded $H^\infty$-calculus and $R$-bisectoriality is studied for some perturbed first order Hodge-Dirac and Dirac type  bisectorial operators in $L^p$ spaces (earlier work on such operators appear  in \cite{Ajiev}). It is known that the former implies  the latter in subspaces of $L^p$ \cite[Theorem 5.3]{KW}. But the converse  is not known. In this specific case, the converse was obtained  but for $p$ in a given open interval, not just  one value of $p$. Subsequently, in \cite{HM}, the $R$-bisectoriality on $L^p$ for these first order  operators was shown to be stable under perturbation of $p$, allowing to apply the above mentioned results and complete the study.  The proof of this result in \cite{HM} uses an extrapolation ``\`a la'' Calder\'on and Zygmund, by real methods. Here, we wish to observe that there is an extrapolation ``\`a la''  {\v{S}}ne{\u\i}berg using complex function theory. Nevertheless, the argument in \cite{HM} is useful to obtain further characterization of $R$-bisectoriality in $L^p$ in terms of    kernel/range decomposition. Indeed, we shall see that for the first order operators in $L^p$ considered in \cite{HM}, this property remains true  by perturbation of $p$ in the same interval as for perturbation of $R$-bisectoriality.  

Our plan is to first review properties of perturbed Dirac type operators at some abstract level of generality. Then we consider the  first order   differential operators of \cite{HMP2, HM}. We next show the  {\v{S}}ne{\u\i}berg extrapolation for ($R$-)bisectoriality of such operators and conclude for the equivalence of $R$-bisectoriality and $H^\infty$-calculus. We then show  that   of $H^\infty$-calculus, $R$-bisectoriality, bisectoriality hold simultaneously to    kernel/range decomposition on a certain open interval. We interpret this with the motivating example coming from a second order differential operator in divergence form, showing that this interval  agrees with an interval studied in \cite{Auscher}. 

\section{Abstract results}

In this section, we assume without mention the followings:  $\mX$ is a reflexive  complex Banach space. The duality between $\mX$ and its dual $\mX^*$  is denoted  $\langle u^*,u\rangle$  and is anti-linear in $u^*$ and linear in $u$.   Next, $D$ is a closed, densely defined operator on $\mX$ and $B$ is a bounded operator on $\mX$.  We state a first proposition on properties of $BD,DB$ and their duals under various hypotheses.

\begin{prop}\label{prop:0}
\begin{enumerate}
  \item $BD$ with $\dom(BD)=\dom(D)$ is densely defined. Its adjoint, $(BD)^*$, is closed, and
  $\dom((BD)^*)=\{u \in \mX\, ; \, B^*u \in \dom(D^*)\}=\dom(D^*B^*)$ with $(BD)^*=D^*B^*$.
  \item Assume that $\|Bu\| \gtrsim \|u\|$ for all $u\in \closran D$. Then,
\subitem (i)   $B|_{\closran D}\colon \closran D \to \closran{BD}$ is an isomorphism, 
\subitem (ii) $BD$ and $D^*B^*$ are both densely defined and closed.  
 \subitem (iii) $DB|_{\closran {D}}$ and $BD|_{\closran {BD}}$ are similar under conjugation  by $B|_{\closran {D}}$.
  \item Assume  that $\|Bu\| \gtrsim \|u\|$ for all $u\in \closran D$ and $\mX= \nul(D) \oplus \closran{D}$. Then $\nul (D)=\nul (BD)$. 
  \item  Assume that $\|Bu\| \gtrsim \|u\|$ for all $u\in \closran D$  and $\mX= \nul(D) \oplus \closran{BD}$. Then, 
  \subitem (i) $\mX= \nul(DB) \oplus \closran D$. 
   \subitem (ii)  $\ran(DB)=\ran(D)$.
  \item  Assume that $\|Bu\| \gtrsim \|u\|$ for all $u\in \closran D$  and $\mX= \nul(D) \oplus \closran{BD}$. Then, 
  \subitem (i) $\mX'= \nul(D^*B^*) \oplus \closran{D^*}$.
  \subitem (ii) $ \big(\,\closran{BD}\,\big)^* = \closran{D^*}$ in the duality $\langle\ , \  \rangle$,
   with comparable norms.
  \subitem  (iii) $\|B^*u^*\| \gtrsim \|u^*\|$ for all $u^*\in \closran {D^*}$, hence $B^*|_{\closran {D^*}}\colon \closran {D^*} \to \closran{B^*D^*}$ is an isomorphism.
  \subitem (iv) $(DB)^*=B^*D^*$.
  \subitem  (v) $D^*B^*|_{\closran {D^*}}$ and $B^*D^*|_{\closran {B^*D^*}}$ are similar under conjugation  by $B^*|_{\closran {D^*}}$.
  \subitem (vi) $ \closran{B^*D^*} = \big(\, \closran{D}\, \big)^*$ in the duality $\langle\ , \  \rangle$,
   with comparable norms.
 \subitem  (vii) $D^*B^*|_{\closran {D^*}}$ is the adjoint of $BD|_{\closran {BD}}$ in the duality $\langle\ , \  \rangle$.
  \subitem (viii) $B^*D^*|_{\closran {B^*D^*}}$ is the adjoint to $DB|_{\closran {D}}$ in the duality $\langle\ , \  \rangle$. 
  
  \end{enumerate}
\end{prop}

\begin{proof} We skip the  elementary proofs of  (1) and (2) except for (2iii). See the proof of \cite[Lemma 4.1]{AKM} where this is explicitly stated on a Hilbert space.  The reflexivity of $\mX$ is used to deduce that $D^*B^*=(BD)^*$ is densely defined. We next show (2iii). Note that $\closran {D}$ is an invariant subspace for $DB$. Let $\beta= B|_{\closran {D}}$. If $u\in \dom(BD|_{\closran {BD}})=  \closran{BD}\cap \dom(BD)=\closran{BD}\cap \dom(D)$, then $\beta^{-1}u\in \closran{D}\cap \dom(D\beta)=\closran{D}\cap \dom(DB)=\dom (DB|_{\closran {D}})$ and 
$$
BD u= \beta Du= \beta (DB)(\beta^{-1}u).
$$

We now prove (3). Clearly $\nul(D)\subset \nul(BD)$. Conversely, let $u\in \nul(BD)$. From $\mX= \nul(D) \oplus \closran{D}$ write $u=v+w$ with $v\in \nul(D)$ and $w\in \closran{D}$.  It follows that $Du=Dw$ and $0=BDu=BDw$. As $B|_{\closran D}\colon \closran D \to \closran{BD}$ is an isomorphism, we have $w=0$. Hence, $u=v\in \nul(D)$. 

We next prove (4). We know  that $DB$ is closed.  Its null space is $\nul(DB)=\{u\in \mX\, ;\, Bu \in \nul(D)\}$. 

Let us first show (i), namely that $\mX= \nul(DB) \oplus \closran D .$ As $\mX= \nul(D) \oplus \closran{BD}$ by assumption,   the projection $P_{1}$ on $\closran{BD}$ along $\nul(D)$ is bounded on $\mX$. Let $u\in \mX$.   
As $P_{1}Bu\in \closran{BD}$, there exists $v\in  \closran D$ such that $P_{1}Bu=Bv$ and $\|v\| \lesssim \|Bv\|= \|P_{1}Bu\|\lesssim \|u\|$.  Since $Bu= (I-P_{1})Bu+ P_{1}Bu$ and 
 $(I-P_{1})Bu\in \nul(D)$, we have $B(u-v)\in \nul(D)$, that is $u-v\in \nul(DB)$. It follows  that $u= u-v+v \in \nul(DB) + \closran D $ with $\|v\|+\|u-v\| \lesssim \|u\|$. 
 
Next, we see that $\ran(DB)=\ran(D)$. Indeed,  the inclusion $\ran(DB)\subseteq\ran(D)$ is trivial. For the other direction, if $v\in \ran(D)$, then one can find $u\in \dom(D)$ such that $v=Du$. Using 
$\mX= \nul(D) \oplus \closran {BD}$, one can select $u\in \closran {BD}= B \closran D$ and write $u= Bw$ with $w\in \closran D$. Hence $v= DBw \in \ran(DB)$.

We turn to the proof of  (5). Item  (i) is  proved as Lemma 6.2 in \cite{HMP2}. To see (ii), we observe that if $u^*\in \closran{D^*}$, then $\closran{BD} \ni u\mapsto \langle u^*, u\rangle$ is a continuous linear functional. Conversely, if $\ell\in \big(\,\closran{BD}\, \big)^*$, then by the Hahn-Banach theorem, there is $u^*\in \mX^*$ such that $\ell(u)= \langle u^*, u\rangle$ for all $u\in \closran{BD}$. Write 
$u^*=v^*+w^*$ with $v^*\in \nul(D^*B^*)$ and $w^*\in \closran{D^*}$ by (i). Since $\langle v^*, u\rangle=0$ for all $u\in \closran{BD}$, we have $\ell(u)= \langle w^*, u\rangle$ for all $u\in \closran{BD}$ with $w^*\in \closran{D^*}$. 

To see (iii), consider again $\beta=B|_{\closran {D}}$. Let $u^*\in \closran{D^*}, u\in  \closran{D}$. Then
$$
\langle B^*u^*, u\rangle = \langle u^*, Bu\rangle = \langle u^*, \beta u\rangle.
$$
Using (ii), we  have proved $B^*|_{\closran {D^*}}= \beta^*$ and the conclusion follows.

To see item (iv),  we remark that combining (iii) and item (2) applied to $B^*D^*$, we have $(B^*D^*)^*=DB$, hence $(DB)^*=B^*D^*$ by reflexivity. 

Item  (v) follows from item (iii) as for item (2iii). 

Item (vi) follows from $\closran D= \beta^{-1} \closran{BD}$ with $\beta$ as above:  $\big(\,\closran D\,\big)^*= \beta^{^*} \big(\,\closran{BD}\,\big)^*$ and we conclude using item (ii) and $B^*|_{\closran {D^*}}= \beta^*$.

Item (vii) follows from the dualities $(BD)^*=D^*B^*$ and   $ \big(\,\closran{BD}\,\big)^* = \closran{D^*}$.

To prove item (viii), we recall that  $DB|_{\closran {D}}= \beta^{-1} BD|_{\closran {BD}}\beta$. Thus using what precedes, 
$$\big(\, DB|_{\closran {D}}\, \big)^* = \beta^* \big(\,BD|_{\closran {BD}}\, \big)^* (\beta^*)^{-1} = B^*(D^*B^*|_{\closran {D^*}})(\beta^*)^{-1}= B^*D^*|_{\closran {B^*D^*}}.$$
\end{proof} 

%

\begin{rem} Note that the property $\|Bu\| \gtrsim \|u\|$ for all $u\in \closran D$ alone does not seem to imply $\|B^*u^*\| \gtrsim \|u^*\|$ for all $u^*\in \closran {D^*}$. Hence the situation for $BD$ and $B^*D^*$ is not completely symmetric without further hypotheses.
\end{rem}

Here is an easy way to check the assumptions  above from    kernel/range decomposes assumptions.

\begin{cor} \label{cor:2} Assume  that $\|Bu\| \gtrsim \|u\|$ for all $u\in \closran D$. If $D$ and $BD$     kernel/range decompose, then $\mX= \nul(D) \oplus \closran{BD}$.  In particular this holds if $D$ and $BD$ are bisectorial. 
\end{cor}

\begin{proof} By Proposition 
\ref{prop:0}, (3), $\nul(D)=\nul(BD)$. We conclude from $\mX= \nul(BD) \oplus \closran{BD}$. 
 \end{proof}

%
%
%

\begin{cor}\label{cor:3} Assume  that $\|Bu\| \gtrsim \|u\|$ for all $u\in \closran D$ and  that $D$     kernel/range decomposes.  If  $BD$     kernel/range decomposes so does $DB$. If $BD$ is bisectorial, so is $DB$, with same angle as $BD$. The same holds if $R$-bisectorial replaces bisectorial everywhere. 
\end{cor}

\begin{proof}   The statement about    kernel/range decomposition is a consequence of Corollary \ref{cor:2} and Proposition \ref{prop:0}, item (4). Assume next that $BD$ is bisectorial and let us show that $DB$ is bisectorial. By Proposition \ref{prop:0}, item (2), $DB|_{\closran {D}}$ and $BD|_{\closran {BD}}$ are similar, thus $DB|_{\closran {D}}$ is bisectorial. Trivially $DB|_{\nul (DB)}=0$ is also bisectorial. 
As $\mX= \nul(DB) \oplus \closran D$ by 
  Corollary \ref{cor:2} and Proposition \ref{prop:0}, item (4), we conclude that $DB$ is bisectorial in $\mX$. 


The proof for $R$-bisectoriality is similar. 
\end{proof}


\begin{rem} The converse   $DB$ ($R$-)bisectorial   implies $BD$  ($R$-)bisectorial seems unclear under the above assumptions on $B$ and $D$, even if $\mX$ is reflexive which we assumed. So it appears that the theory is not completely symmetric  for $BD$ and for $DB$ under such assumptions. 
\end{rem}

\begin{cor}\label{cor:4} Assume that  $D$     kernel/range decomposes. The followings are equivalent:
\begin{enumerate}
  \item $\|Bu\| \gtrsim \|u\|$ for all $u\in \closran D$ and   $BD$  bisectorial in $\mX$.
  \item $\|B^*u^*\| \gtrsim \|u^*\|$ for all $u^*\in \closran {D^*}$ and   $B^*D^*$  bisectorial in $\mX^*$.
\end{enumerate}
Moreover the angles are the same. If either one holds, then $DB$ and $D^*B^*$ are also bisectorial, with same angle. The same holds with $R$-bisectorial replacing bisectorial everywhere if $\mX$ is  an $L^p$ space with  $\sigma$-finite measure and $1<p<\infty$.   
\end{cor}

\begin{proof} It is enough to assume (1) by symmetry (recall that we assume $\mX$ reflexive). That $\|B^*u^*\| \gtrsim \|u^*\|$ for all $u^*\in \closran {D^*}$ follows from Corollary \ref{cor:2} and Proposition \ref{prop:0}, item (5). Next, as  $B^*D^*=(DB)^*$ by Proposition \ref{prop:0}, item (5), and as  $DB$ is bisectorial by Corollary \ref{cor:3}, $B^*D^*$ is also bisectorial by general theory. This proves the equivalence. Checking details, one sees that the angles are the same. Bisectoriality of $DB$ and $D^*B^*$ are already used in the proofs. The proof is the same for $R$-bisectoriality, which is stable under taking adjoints on reflexive $L^p$ space with $\sigma$-finite measure (see \cite[Corollary 2.1]{KunW}).
\end{proof}

\section{First order constant coefficients differential systems}

Assume now that $D$ is  a first order differential operator on $\R^n$ acting on functions valued in $\C^N$ whose symbol satisfies the conditions (D0), (D1) and (D2) in \cite{HM}. We do not assume that $D$ is self-adjoint. Let $1<q<\infty$ and $\dom_{q}(D)=\{u\in L^q\, ; \, Du\in L^q\}$ with $L^q:=L^q(\R^n; \C^N)$ and $D_{q}=D$ on $\dom_{q}(D)$.  We keep using the notation $D$ instead of $D_{q}$ for simplicity.   The followings properties have been  shown in \cite{HMP2}.
\begin{enumerate}
  \item $D$ is a $R$-bisectorial operator with $H^\infty$-calculus in $L^q$.
\item $L^q=N_{q}(D)\oplus \clos{\ran_{q}(D)}$.
  \item $N_{q}(D)$ and $\clos{\ran_{q}(D)}$, $1<q<\infty$, are complex interpolation families.
  \item $D$ has the coercivity condition
  $$
  \|\nabla u \|_{q} \lesssim \| Du\|_{q}\quad \mathrm{for\ all\ } u \in \dom_{q}(D)\cap \clos{\ran_{q}(D)} \subset W^{1,q}.$$ Here, we use the notation $\nabla u$ for $\nabla \otimes u$. 
  \item The same properties hold for $D^*$.
\end{enumerate}

Let us add one more property.

\begin{prop}\label{prop:domains} Let $t>0$.
The spaces $\dom_{q}(D)$, $1<q<\infty$, equipped with the  norm $|||f|||_{q,t}:=\|f\|_{q}+t\|Df\|_{q}$, form a complex interpolation family. The same holds for $D^*$.
\end{prop}

\begin{proof} Since $D$ is bisectorial in $L^q$, we have $\|(1+itD)^{-1}u\|_{q} \le C\|u\|_{q}$ with $C$ independent of $t$. [To be precise, we should write $D_{q}$ for $q$ and use that the resolvents are compatible for different values of $q$, that is the resolvents for different $q$ agree on the intersection  of  the $L^q$'s.] Thus
$
(I+itD)^{-1}\colon (L^q, \|\ \|_{q}) \to (\dom_{q}(D), |||\ |||_{q,t})$ is an isomorphism with uniform bounds with respect to $t$: $$\|u\|_{q}\le \|(I+itD)^{-1}u\|_{q}+ t\|D(I+itD)^{-1}u\|_{q}\le (C+1) \|u\|_{q}.$$ The conclusion follows by the fonctoriality of complex interpolation. 
\end{proof}

\section{Perturbed first order differential systems}

Let $B\in L^\infty(\R^n; \mL(\C^N))$. Identified with the operator of multiplication by $B(x)$, $B\in \mL(L^q)$ for all $q$. Its adjoint $B^*$ has the same property. With $D$ as before, introduce the set
$$
\mI(BD)=\{q\in (1,\infty)\, ;\, \|Bu\|_{q} \gtrsim \|u\|_{q} \ \mathrm{for\ all\ } u\in \ran_{q}(D)\}.
$$
By density, we may replace $\ran_{q}(D)$ by its closure.  For $q\in \mI(BD)$, $B|_{\clos{\ran_{q}(D)}}: \clos{\ran_{q}(D)} \to
\clos{\ran_{q}(BD)}$ is an isomorphism. Let
$$b_{q}=\inf  \left(\frac{\|Bu\|_{q} }{ \|u\|_{q}}\, ;\, u \in \clos{\ran_{q}(D)}, u\ne 0\right)>0.$$

\begin{lem} The set $\mI(BD)$ is open.
\end{lem}

\begin{proof}    We  have for all $1<q<\infty$, 
$\|Bu\|_{q}\le \|B\|_{\infty}\|u\|_{q}$. Thus, the bounded map $B\colon \clos{\ran_{q}(D)} \to L^q$ is bounded below by $b_{q}$ for each $q\in \mI(BD)$. Using that $\clos{\ran_{q}(D)}$ and $L^q$ are complex interpolation families,  the result follows from a result of  {\v{S}}ne{\u\i}berg \cite{Snei} (see also Kalton-Mitrea \cite{KalMit}).
\end{proof}

\begin{rem}
 If $B$ is invertible in $ L^\infty(\R^n; \mL(\C^N))$, then $B$ is invertible in $\mL(L^q)$ and its inverse is the operator of multiplication by $B^{-1}$. In this case, $\mI(BD)=(1,\infty)$. 
\end{rem}

For next use, let us recall the statement of  {\v{S}}ne{\u\i}berg (concerning lower bound) and Kalton-Mitrea (concerning invertibility even in the quasi-Banach case). 

\begin{prop}\label{prop:sneiberg} Let $(X_{s})$ and $(Y_{s})$ be two complex interpolation families of Banach spaces for $0<s<1$. Let $T$ be an operator with $C= \sup_{0<s<1} \|T\|_{\mL(X_{s}, Y_{s})}<\infty$. Assume that for  $s_{0}\in (0,1)$ and $\delta >0$,  $\|Tu\|_{Y_{s_{0}}} \ge \delta \|u\|_{X_{s_{0}}}$ for all $u\in X_{s_{0}}$. Then, there is an interval $J$ around $s_{0}$ whose length is bounded below by a number depending  on $C, \delta, s_{0}$ on which $\|Tu\|_{Y_{s}} \ge \frac \delta 2 \|u\|_{X_{s}}$ for all $u\in X_{s}$ and all $s\in J$.  If, moreover, $T$ is invertible in $\mL(X_{s_{0}}, Y_{s_{0}})$ with lower bound $\delta $ then there is an interval $J$ whose length is bounded below by a number depending  on $C, \delta, s_{0}$ on which $T$ is invertible with  $\|Tu\|_{Y_{s}} \ge \frac \delta 2 \|u\|_{X_{s}}$ for all $u\in X_{s}$ and all $s\in J$.  
\end{prop}

Our point  is that the lower bound on the size of $J$ is universal for complex families.   Define two more sets related to the operator $BD$: 
\begin{align*} \mB(BD)&=\{ q \in \mI(BD)\, ; \, BD \ \mathrm{bisectorial\ in \ } L^q\}
\\
\mR(BD)&=\{ q \in \mI(BD)\, ; \, BD \  R\mathrm{-bisectorial\ in \ } L^q\}
\end{align*}
Note that these are subsets of $\mI(BD)$.  We can define the analogous sets for $B^*D^*$.

\begin{prop}\label{prop:dual} Let $1<p<\infty$. Then $p\in \mR(BD)$ (resp. $\mB(BD)$) if and only if $p'\in \mR(B^*D^*)$ (resp. $\mB(B^*D^*)$). 
\end{prop}

\begin{proof} This is Corollary \ref{cor:4}. \end{proof}

Our next results are the new observation of this paper, simplifying the approach of \cite{HM}. 

\begin{prop} \label{prop:6}
These sets  are open. 
\end{prop}

\begin{proof} Let us consider the openness of $\mB(BD)$ first. 
We know that for all $p\in \mI(BD)$, $BD$ is densely defined and closed on $L^p$ from Proposition \ref{prop:0}, item (2). Fix $q \in \mB(BD)$. Let $\omega$ be the angle of bisectoriality in $L^q$ and $\omega<\mu<\pi/2$.  Let $C_{\mu}= \sup_{\lambda\notin S_{\mu}} \|(I+\lambda BD)^{-1}\|_{\mL(L^q)}$. Fix  $\lambda\notin S_{\mu}$. Then
$\|\lambda B D(I+\lambda BD)^{-1}\|_{\mL(L^q)} \le C_{\mu}+1$. Thus for all $u \in \dom_{q}(D)$, 
\begin{align*}
\|(I+\lambda BD)u\|_{q}&\ge (2C_{\mu})^{-1} \|u\|_{q}+ (2C_{\mu}+2)^{-1} |\lambda| \|BD u \|_{q}
\\
& \ge (2C_{\mu})^{-1} \|u\|_{q}+ (2C_{\mu}+2)^{-1}b_{q} |\lambda| \|D u \|_{q}
\\
&
\ge \delta |||u|||_{q,|\lambda|}
\end{align*}
with $\delta = \inf ((2C_{\mu})^{-1}, (2C_{\mu}+2)^{-1}b_{q})>0$.
Also 
$$
\|(I+\lambda BD)u\|_{q} \le C|||u|||_{q,|\lambda|}
$$
with $C= \sup(1, \|B\|_{\infty})$. Applying Proposition \ref{prop:sneiberg} thanks to Proposition \ref{prop:domains}, we obtain  an open interval $J$ about $q$ contained in  $\mI(BD)$ such that for all $\lambda\notin S_{\mu}$ and $p\in J$, $(I+\lambda BD)^{-1}$ is bounded on $L^p$ with bound $2/\delta $. 

The proof for perturbation of $R$-bisectoriality is basically the same, with $C_{\mu}$ being the $R$-bound of $(I+\lambda BD)^{-1}$, that is the best constant in the inequality
$$
\left\|\left(\sum_{j=1}^k|(I+\lambda_{j} BD)^{-1}u_{j}|^2\right)^{1/2}\right\|_{q} \le C
\left\|\left(\sum_{j=1}^k|u_{j}|^2\right)^{1/2}\right\|_{q}
$$
for all $k\in \N$, $\lambda_{1}, \ldots, \lambda_{k}\notin S_{\mu}$ and $u_{1}, \ldots, u_{k}\in L^q$. One works in the sums 
$L^q \oplus \cdots \oplus L^q$ equipped with the norm of the right hand side and $\dom_{q}(D) \oplus \cdots \oplus \dom_{q}(D)$ equipped with 
$$
\left\|\left(\sum_{j=1}^k|u_{j}|^2\right)^{1/2}\right\|_{q} + \left\|\left(\sum_{j=1}^k|\lambda_{j}|^2|Du_{j}|^2\right)^{1/2}\right\|_{q}.
$$
To obtain the $R$-lower bound (replacing $\delta $), one linearizes using the Kahane-Kintchine inequality with the Rademacher functions
$$
\left(\sum_{j=1}^k|u_{j}|^2\right)^{1/2} \sim \left(\int_{0}^1 \bigg| \sum_{j=1}^k r_{j}(t)u_{j}\bigg|^q\, dt\right)^{1/q},
$$
valid for any $q\in (1,\infty)$ (see, for example, \cite{KunW} and follow the argument  above). Details are left to the reader.
\end{proof}

\begin{rem}\label{rem:interval} These sets may not be intervals. They are (possibly empty) intervals  when restricted to each connected component of $\mI(BD)$ because ($R$-)bisectoriality interpolates in $L^p$ scales. See \cite[Corollary 3.9]{KKW} for a proof concerning $R$-bisectoriality.  In particular, if $\mI(BD)=(1,\infty)$ these sets are (possibly empty) open intervals. 
\end{rem}

\begin{thm}\label{thm:equiv} For $p\in \mI(BD)$, the following assertions are equivalent:
\begin{itemize}
  \item[(i)] $BD$ is $R$-bisectorial  in  $L^p$.
  \item[(ii)] $BD$ is bisectorial and has an $H^\infty$-calculus  in $L^p$. 
\end{itemize}
Moreover, the angles in (i) and (ii) are the same. Furthermore, if one of the items holds, then they hold as well for $DB$, and also for $B^*D^*$ and $D^*B^*$ in $L^{p'}$.   
\end{thm}

\begin{proof}
The implication (ii) $\Rightarrow$ (i) is a general fact proved in \cite{KW}. Assume conversely that (i) holds. Then, there is an interval  $(p_{1},p_{2})$ around $p$ for which (i) holds with the same angle by Proposition \ref{prop:6}. Note also that (2) and (3) of Proposition \ref{prop:0} apply with $\mX=L^q$ for each $q\in (p_{1},p_{2})$. Hence,  $B^*$ has a lower bound on $\clos{\ran_{q'}(D^*)}$. We may apply Corollary 8.17 of \cite{HMP2}, which states that $D^*B^*$ satisfies (ii) on $L^{q'}$. By duality, we conclude that $BD$ satisfies (ii) in $L^q$. 

The last part of the statement now follows from Corollary \ref{cor:4}.
\end{proof}

\begin{rem} As  $p\in \mR(BD)$ if and only if $p\in \mR(B^*D^*)$,   Proposition \ref{prop:6} and  Theorem \ref{thm:equiv} can be compared to  Theorem 2.5 of \cite{HM} for the stability of $R$-bisectoriality and the equivalence with $H^\infty$-calculus. The argument here   is much easier and fairly general once we have Proposition \ref{prop:domains}. However, the argument in \cite{HM} is useful since it contains a quantitative estimate on how far one can move from $q$. We come back to this below.  Recall that  the motivation of \cite[Theorem 2.5]{HM}, thus reproved here,  is to complete the theory developed in \cite{HMP2}. 
\end{rem}

\section{Relation to  kernel/range decomposition}

For a closed unbounded operator $\mA$ on a Banach space $\mX$, recall that $\mA$    kernel/range decomposes if $\mX=\nul(\mA)\oplus \closran \mA$ and that it is implied by  bisectoriality. The converse is not true (the shift on $\ell^2(\Z)$ is invertible, so the    kernel/range decomposition is trivial, but it is not bisectorial as its spectrum is the unit circle).  For the class of $BD$ operators in the previous section, we shall show that a converse holds. 

For a set $A\subseteq (1,\infty)$, let $A'=\{q'\, ;\, q\in A\}$.

Consider $D$ and $B$ as in Section 4. 
Recall that $p\in\mR(BD)$ if and only if $p'\in \mR(B^*D^*)$. That is $\mR(B^*D^*)'=\mR(BD)$. Recall also that $\mR(BD)\subseteq \mI(BD)$, hence  $\mR(BD)\subseteq \mI(B^*D^*)'$ as well. 

Assume $p_{0}\in \mR(BD)$ and let $I_{0}$ be the connected component of 
$\mI(BD)\cap \mI(B^*D^*)'$ that contains $p_{0}$. It is an open interval.  

Let 
\begin{align*} \mB_{0}(BD)&=\{ q \in I_{0}\, ; \, BD \ \mathrm{bisectorial\ in \ } L^q\}
\\
\mR_{{0}}(BD)&=\{ q \in I_{0}\, ; \, BD \  R\mathrm{-bisectorial\ in \ } L^q\}
\\
 \mH_{0}(BD)&=\{ q \in I_{0}\, ; \, BD \ \mathrm{bisectorial\ in \ } L^q \mathrm{\ with \ } H^\infty\mathrm{-calculus} \}
 \\
 \mS_{{0}}(BD)&=cc_{p_{0}}\{ q \in I_{0}\, ; \, BD \  \mathrm{kernel/range  \ decomposes\ in\ } L^q\}
\end{align*}
The notation $cc_{p_{0}}$ means the connected component that contains $p_{0}$. 

\begin{thm}\label{thm:equal} The four sets above are  equal open intervals.
\end{thm}

It should be noted that the theorem assumes non emptyness of $\mR_{{0}}(BD)$.

\begin{proof} It is clear that $ \mH_{0}(BD)\subseteq \mR_{{0}}(BD) \subseteq  \mB_{0}(BD)$.
 By  Proposition \ref{prop:6} and the discussion in Remark \ref{rem:interval},   $\mR_{{0}}(BD)$ and $ \mB_{0}(BD) $ are open subintervals of $I_{0}$. By Theorem \ref{thm:equiv}, we also know that $\mH_{0}(BD)= \mR_{{0}}(BD)$.
 
As bisectoriality implies    kernel/range decomposition,  $\mB_{0}(BD)$ is contained in the set $ \{ q \in I_{0}\, ; \, BD \ \mathrm{kernel/range  \ decomposes\ in\ } L^q\}$.   As  $\mB_{0}(BD)$ contains $p_{0}$, we have $\mB_{0}(BD) \subseteq \mS_{0}(BD)$.    Thus it remains to show that $\mS_{{0}}(BD) \subseteq \mR_{{0}}(BD)$, which is done in the next results.
\end{proof}  

For $1<p<\infty$, let $p^*, p_{*}$ be the upper and lower Sobolev exponents: $p^*= \frac{np}{n-p}$ if $p<n$ and $p^*=\infty$ if $p\ge n$, while $p_{*}= \frac{np}{n+p}$.

\begin{lem} Let $p\in  \mR_{{0}}(BD)$. Then $BD|_{\clos{\ran_{q} ({BD})}}$ is $R$-bisectorial (in $\clos{\ran_{q}( {BD})}$) for $q\in I_{0}\cap (p_{*}, p^*)$. 
\end{lem}

\begin{proof} The (non-trivial) argument to extrapolate $R$-bisectoria\-lity at $p$ to $R$-bisectoria\-lity at any $q\in I_{0}\cap (p_{*},p)$ is exactly what is proved  in  Sections 3 and 4 of   \cite{HM}, taken away the arguments related to    kernel/range decomposition which are not assumed here. 
We next provide the argument for $q\in I_{0}\cap (p, p^*)$.  By duality, $p'\in \mR(B^*D^*)$. By symmetry of the assumptions, $B^*D^*|_{\clos{\ran_{q'} ({B^*D^*})}}$ is $R$-bisectorial. By duality of $R$-bisectoriality in subspaces of reflexive Lebesgue spaces and Proposition \ref{prop:0}, item (5), $DB|_{\clos{\ran_{q} ({D})}}$ is $R$-bisectorial. By Proposition \ref{prop:0}, item (2),  this implies that $BD|_{\clos{\ran_{q} ({BD})}}$ is $R$-bisectorial. 
\end{proof}

\begin{cor} $\mS_{{0}}(BD) \subseteq \mR_{{0}}(BD)$.
\end{cor}

\begin{proof} The set $\{ q \in I_{0}\, ; \, BD \ \mathrm{kernel/range\  decomposesg\ in\ } L^q\}$  is open (this was observed in \cite{HM}, again as a consequence of   {\v{S}}ne{\u\i}berg's result). Thus, as a connected component, $\mS_{{0}}(BD)$ is an open interval. 
Write $\mR_{{0}}(BD)=(r_{-},r_{+})$ and $\mS_{{0}}(BD)=(s_{-},s_{+})$ and recall that $(r_{-},r_{+}) \subseteq (s_{-},s_{+})$. Assume $s_{-}<r_{-}$. One can find $p,q$ with $q\in I_{0}\cap (p_{*}, p)$ and $s_{-}<q\le r_{-}<p<r_{+}$.  By the previous lemma, we have that  $BD|_{\clos{\ran_{q} ({BD})}}$ is $R$-bisectorial in $\clos{\ran_{q}( {BD})}$. Also $BD|_{\nul_{q} ({BD})}=0$ is $R$-bisectorial. As $q\in \mS_{{0}}(BD)=(s_{-},s_{+})$, we have $L^q=\clos{\ran_{q}( {BD})}\oplus \nul_{q}(BD)$. Hence, $BD$ is $R$-bisectorial in $L^q$. This is a contradiction as $q\notin \mR_{{0}}(BD)$. Thus $r_{-}\le s_{-}$. The argument to obtain $s_{+}\le r_{+}$ is similar. 
\end{proof}

\begin{rem} It was observed and heavily used in \cite{HM} that for a given $p$, $L^p$ boundedness of the resolvent of $BD$ self-improves to off-diagonal estimates.  Thus, the set of those $p\in I_{0}$ for which one has such estimates in addition to bisectoriality in $L^p$ is equal to $\mB_{0}(BD)$ as well. 

\end{rem}

\section{Self-adjoint $D$ and accretive $B$}

The operators $D$ and $B$ are still as in Section 4.   In addition, 
assume that $D$ is self-adjoint on $L^2$ and  that 
 $B$ is strictly accretive in $\clos{\ran_{2}(D)}$, that is for some $\kappa>0$, 
$$
\re \langle u, Bu \rangle \ge \kappa \|u\|_{2}^2, \quad \forall\ u \in \clos{\ran_{2}(D)}.
$$
Then, $B$ and $B^*$ have lower bound $\kappa$ on $\clos{\ran_{2}(D)}$ and $\clos{\ran_{2}(D^*)}=\clos{\ran_{2}(D)}$. In this case, $BD$ and $DB=(B^*D)^*$ (replacing $B$ by $B^*$) are bisectorial operators in $L^2$. Moreover, using that $B$ is multiplication and $D$ a coercive first order differential operator with constant coefficients,  \cite[Theorem 3.1]{AKM} (see \cite{AAM} for a direct proof) shows that
$BD$ and $DB$ have $H^\infty$-calculus in $L^2$.  Thus,  Theorem \ref{thm:equal} applies and
one has the

\begin{thm}
There exists  an open interval $I(BD)=(q_{-}(BD),q_{+}(BD))\subseteq (1,\infty)$, containing $2$,  with the following dichotomy:        $H^\infty$-calculus, $R$-bisectoriality,   bisectoriality  and    kernel/range decomposition hold for $BD$  in $L^p$ if $p\in I(BD)$ and all fail if $p=q_{\pm}(BD)$.  The same property hold for $DB$ with  $I(DB)=I(BD)$.  The same property hold for $B^*D$ and $DB^*$ in the dual interval $I(DB^*)=I(B^*D)=(I(BD))'$. 
\end{thm} 

In applications, one tries to find an interval of $p$ for bisectoriality, which is the easiest property to check. 

The example that motivated the study of perturbed Dirac operators it the following setup, introduced in \cite{AMN} and exploited in \cite{AKM}  to reprove the Kato square root theorem obtained in \cite{AHLMT} for second order operators and in \cite{AHMT} for systems. 
Let $A\in L^{\infty}(\R^n;\mL(\C^m\otimes\C^n))$ satisfy
\begin{equation}
  \int_{\R^n}\nabla\bar{u}(x)\cdot A(x)\nabla u(x) dx\gtrsim\Norm{\nabla u}{2}^2,
\end{equation}
for all $u\in W^{1,2}(\R^n;\C^m)$. Then $BD$, with $\displaystyle B=\begin{pmatrix} I & 0 \\ 0 & A \end{pmatrix}$ and $\displaystyle D=\begin{pmatrix} 0 & -\divv \\ \nabla & 0 
\end{pmatrix}$, has a bounded $H^{\infty}$-calculus in $L^p(\R^n;\C^m\oplus[\C^m\otimes\C^n])$ for all $p\in(q_{-}(BD),q_{+}(BD))$, with angle at  most equal to the accretivity angle of $A$. 

Let us finish with the interpretation of the    kernel/range decomposition in this particular example. As 
$\displaystyle BD=\begin{pmatrix} 0 & -\divv \\ A\nabla & 0 
\end{pmatrix}$, we see that 
$$\nul_{p}(BD)=\{u=(0,g)\in L^p(\R^n;\C^m\oplus[\C^m\otimes\C^n])\, ;\, 
\divv g=0\}$$ and 
$$\clos{\ran_{p}(BD)}=\{u=(f,g)\in L^p(\R^n;\C^m\oplus[\C^m\otimes\C^n])\, ;\, 
 g=A\nabla h, h\in \dot W^{1,p}(\R^n;\C^m) \},$$
where $\dot W^{1,p}(\R^n;\C^m)$ is the homogeneous Sobolev space. Thus,
 \begin{equation}
L^p(\R^n;\C^m\oplus[\C^m\otimes\C^n])=\nul_{p}(BD)\oplus  \clos{\ran_{p}(BD)}
\end{equation}
  is equivalent to the \emph{Hodge splitting adapted to $A$ for vector fields}
\begin{equation}\label{eq:hodge1}
L^p(\R^n;\C^m\otimes\C^n) = \nul_{p}(\divv) \oplus A\nabla  \dot W^{1,p}(\R^n;\C^m).
\end{equation}
 Writing details for $DB$ instead we arrive the equivalence between
 \begin{equation}
L^p(\R^n;\C^m\oplus[\C^m\otimes\C^n])=\nul_{p}(DB)\oplus  \clos{\ran_{p}(DB)}
\end{equation}
and a second  \emph{Hodge splitting adapted to $A$ for vector fields}
 \begin{equation}\label{eq:hodge2}
L^p(\R^n;\C^m\otimes\C^n) = \nul_{p}(\divv A) \oplus \nabla  \dot W^{1,p}(\R^n;\C^m).
\end{equation}
As $q_{\pm}(BD)=q_{{\pm}}(DB)$, we obtain that \eqref{eq:hodge1} and \eqref{eq:hodge2} hold for $p\in (q_{-}(BD), q_{+}(BD))$ and fail at the endpoints. 

Let $L= -\divv A\nabla$. It was shown in \cite[Corollary 4.24]{Auscher} that \eqref{eq:hodge2} holds for $p\in (q_{+}(L^*)', q_{+}(L))$,  where the number  $q_{+}(L)$ is defined as the supremum of those $p>2$ for which $t^{1/2}\nabla e^{-tL}$ is uniformly bounded on $L^p$ for $t>0$ (Strictly speaking, this is done when $m=1$, and Section 7.2 in \cite{Auscher} gives an account of the extension to systems). As a consequence, we have shown that  $q_{+}(BD)=q_{{+}}(DB)=q_{+}(L)$ and $q_{-}(BD)=q_{{-}}(DB)=q_{+}(L^*)'$.

In the previous example, the matrix $B$ is block-diagonal. If $B$ is a full matrix, then $DB$ and $BD$ happen to be in relation with a second order system in $\R^{n+1}_{+}$ as first shown in \cite{AAMarkiv}. Their study brought new information to the boundary value problems associated to such systems when $p=2$. Details when $p\ne 2$  will appear in the forthcoming PhD thesis of the second author. 

\section{Acknowledgments} This work is part of the forthcoming PhD thesis of the second author. The first author thanks the organizers of the IWOTA 2012 conference in Sydney for a stimulating environment. Both authors were partially supported by  the ANR project ``Harmonic Analysis at its Boundaries``, ANR-12-BS01-0013-01.

\def\cprime{$'$}

\end{document}